\documentclass[ECP, preprint]{ejpecp} 

\usepackage[T1]{fontenc}
\usepackage[utf8]{inputenc}
\usepackage{xcolor}
\usepackage{pgfplots}
\pgfplotsset{compat=1.17}

\SHORTTITLE{Limit shape formulas for a generalized Sepp\"al\"ainen--Johansson model}

\TITLE{Limit shape formulas for a generalized Sepp\"al\"ainen--Johansson model
   } 

\AUTHORS{%
  Julian~Ransford\footnote{University of Toronto, Ontario, Canada.
    \EMAIL{julian.ransford@mail.utoronto.ca}}
 }


\AMSSUBJ{60K37; 60K35} 

\SUBMITTED{January 11, 2024} 
\ACCEPTED{} 

\VOLUME{0}
\YEAR{2020}
\PAPERNUM{0}
\DOI{10.1214/YY-TN}

\ABSTRACT{We consider a simplified model of first-passage percolation, involving two families of i.i.d. random variables $\{\xi_{ij}\}$ and $\{\eta_{ij}\}$ corresponding to the weights of the horizontal and vertical edges respectively. We obtain an explicit formula for the limiting shape of the first-passage distance expressed in terms of the corresponding limit shapes of the two sets of weights for the Sepp\"al\"ainen--Johansson model. We also study the limiting fluctuations of this model when at least one of the sets of weights is Bernoulli distributed.}

\newcommand \RR{\mathbb{R}}

\newcommand \PP{\mathbb{P}}

\newcommand \ZZ{\mathbb{Z}}
\newcommand \EE{\mathbb{E}}
\newcommand \Dc{\mathcal{D}}
\newcommand \Ec{\mathcal{E}}
\newcommand \Ac{\mathcal{A}}
\newcommand \Fc{\mathcal{F}}
\newcommand \ep{\epsilon}
\newcommand \om{\omega}
\newcommand \tom{\widetilde{\omega}}

\begin{document}



\section{Introduction}

Let $B_{ij}, i,j \ge 0$ be i.i.d Bernoulli($p$) random variables, and $\xi_{ij}, \eta_{ij} \, i,j \ge 0$ be families of independent random variables. We assume that the $\xi_{ij}$'s are non-negative, integrable and have a common distribution, and likewise for the $\eta_{ij}$'s (though the distribution of the $\xi_{ij}$'s can be different of that of the $\eta_{ij}$'s). We place weights on the edges $e$ of $\ZZ_{\ge 0}^2$ as follows.
\begin{itemize}
\item If $e$ is the horizontal edge joining $(i-1,j)$ to $(i,j)$, then $e$ has weight $\om_e=B_{ij}\xi_{ij}$.
\item If $e$ is the vertical edge joining $(i-1,j)$ to $(i,j)$, then $e$ has weight $\om_e=(1-B_{ij})\eta_{ij}$.
\end{itemize}
So for any vertex $(i,j)$, one of the two ``incoming'' edges $(i-1,j) \to (i,j)$ or $(i,j-1) \to (i,j)$ will have weight 0. Given an up-right path $\pi$ (i.e a sequence of adjacent vertices which only goes up or right), we define its weight $S(\pi)$ as
\[S(\pi)=\sum_{e \in \pi} \om_e\]
where the sum is taken over the edges that $\pi$ traverses; see Figure \ref{fig: path}. For two points $(a,b)$ and $(m,n)$ with $a \le m$ and $b \le n$, we define the \emph{first-passage value} from $(a,b)$ to $(m,n)$ as
\begin{equation}\label{eq: fp def}
F(a,b;m,n)=\min_{\pi: (a,b) \to (m,n)} S(\pi)
\end{equation}
where the minimum is taken over all up-right paths $\pi$ started at $(a,b)$ and finishing at $(m,n)$. We write $F(m,n)$ for $F(0,0;m,n)$. A path which achieves the minimum in \eqref{eq: fp def} is called a \emph{geodesic}.

The function $F$ defines a \emph{directed metric} on $\ZZ_{\ge 0}^2 \times \ZZ_{\ge 0}^2$, in the sense that for any point $(m,n)$, we have $F(m,n;m,n)=0$, and $F$ satisfies the triangle inequality, but the distance can only be measured in one direction; $F(a,b;m,n)$ is only defined when $a \le m$ and $b \le n$.

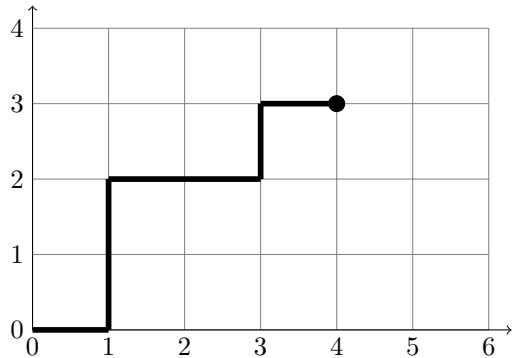
\begin{figure}
\centering
	\begin{tikzpicture}
	\draw[very thin,color=gray] (0,0) grid (6,4);
	\draw[->] (0,0) -- (6.3,0) node[right] {};
	\draw[->] (0,0) -- (0,4.3) node[above] {};
	
	\foreach \i in {0,1,...,6}{
	\coordinate[label=below:$\i$] (x) at (\i,0);
	}
	
	\foreach \i in {0,1,...,4}{
	\coordinate[label=left:$\i$] (x) at (0,\i);
	}
	
	\filldraw[black]  (4,3) circle (3pt);
	
	\draw[-,color=black,line width=0.75mm] (0,0) -- (1,0) node[right]{};
	\draw[-,color=black,line width=0.75mm] (1,0) -- (1,2) node[right]{};
	\draw[-,color=black,line width=0.75mm] (1,2) -- (3,2) node[right]{};
	\draw[-,color=black,line width=0.75mm] (3,2) -- (3,3) node[right]{};
	\draw[-,color=black,line width=0.75mm] (3,3) -- (4,3) node[right]{};	
	
	\end{tikzpicture}
\caption{An upright path from (0,0) to (4,3). The weight of this path is $B_{1,0}\xi_{1,0}+(1-B_{1,1})\eta_{1,1}+(1-B_{1,2})\eta_{1,2}+B_{2,2}\xi_{2,2}+B_{3,2}\xi_{3,2}+(1-B_{3,3})\eta_{3,3}+B_{4,3}\xi_{4,3}$.}
\label{fig: path}
\end{figure}

The special case when all the $B_{ij}$'s are equal to one (i.e only horizontal edges have non-zero weights) is known as the Sepp\"al\"ainen--Johansson (SJ) model. It was introduced by Sepp\"al\"ainen in \cite{sepp98} as a simplified model of directed first-passage percolation, and Johansson showed in \cite{johan01} that in the special case of Bernoulli weights, the model is completely solvable. Indeed, the law of the last-passage value (that is where we take maximum instead of minimum in \eqref{eq: fp def}) is the same as that of the top point of the Krawtchouk ensemble, a discrete orthogonal polynomial ensemble. 

By Kingman's subadditive ergodic theorem, there exists a deterministic function $f$ on $\RR_{\ge 0}^2$ (which depends on the distribution of the weights) such that 
\[\frac{F(\lfloor nx \rfloor, \lfloor ny \lfloor)}{n} \to f(x,y)\]
almost surely for all $x,y \ge 0$. We will refer to $f$ as the \emph{limit shape}. It follows from the translation invariance of this model that $f$ must be homogeneous (i.e $f(cx,cy)=cf(x,y)$ for all $c \ge 0$), and together with the triangle inequality for $F$, we have that $f$ also satisfies a triangle inequality:
\[f(x_1+x_2,y_1+y_2) \le f(x_1,y_1)+f(x_2,y_2).\]
The homogeneity plus triangle inequality clearly imply that $f$ must be convex, and so in particular it is continuous.

Sepp\"al\"ainen obtained in \cite{sepp98} the following explicit formula for the limit shape in the SJ model with Bernoulli($p$) weights:
\begin{equation}\label{eq: Bernoulli formula}
f(x,y)=
\begin{cases}
(\sqrt{px}-\sqrt{(1-p)y})^2 & \quad \text{if } x \ge \frac{1-p}{p}y\\
0 & \quad \text{otherwise}.
\end{cases}
\end{equation}
Similar formulas are known in the SJ model for geometric and exponential weights; see \cite{oco99} for example. The most general cases for which we can compute explicitly the limit shape are when $\xi_{ij}\overset{d}{=}B \cdot X$, where $B$ is Bernoulli distributed, and $X$ is a geometric or exponential independent of $B$; see \cite{mart09}.

The generalized SJ model that we are considering first appears in \cite{bar-cor17} for $\xi_{ij}$ and $\eta_{ij}$ exponentially distributed as the zero temperature limit of the beta polymer. The authors show that the distribution function of a point to half-line first-passage percolation version of this model can be written down explicitly as a Fredholm determinant, and as a result, they obtain Tracy--Widom limiting fluctuations.

Much more is known about the SJ model for Bernoulli weights. Indeed, in \cite{dauv-nica-vir23}, it was shown that the scaled fluctuations of the first-passage value function converges in distribution with respect to uniform convergence on compact sets to the Airy$_2$ process. Using this, we obtain, in the Bernoulli case, the limiting distribution of the scaled fluctuations of $F(nx,ny)$, for points satisfying $x \neq (1-p)y/p$. 

Along the line $x=(1-p)y/p$, the fluctuations of $F$ behave very differently. If we scale space diffusively, then the fluctuations of $F$ are of order 1, and $F$ has a completely different scaling limit: the Brownian web distance. To see why the Brownian web is involved, consider the subgraph of $\ZZ^2$ obtained by only keeping the edges $e$ whose corresponding weight $\om_e$ is 0. The resulting subgraph is then a system of coalescing random walks moving in the south-west direction; see Figure \ref{fig: RW web}. To each removed edge, we associate a cost of jumping across it given by the weight of that edge. Then $F(a,b;m,n)$ is precisely the smallest total cost one has to pay while traversing through the random walk web to get from $(a,b)$ to $(m,n)$. In the special case where the $\xi_{ij}'s$ and $\eta_{ij}'s$ are all equal to 1, this is simply the smallest number of jumps one has to make between the different random walks. The diffusive limit of the random walk web is the Brownian web, and so $F$ converges to a directed metric on the Brownian web. This is proven in \cite{vet-vir23}. See \cite{fon-iso-new-rav04} and the references therein for a survey on the Brownian web.
\begin{figure}
\centering
\begin{tikzpicture}[scale=0.8]
\draw[color=gray, step=0.5, line width=0.3mm, dotted] (0,0) grid (7.5,7.5);
	\draw[-] (0,0) -- (7.8,0) node[right] {};
	\draw[-] (0,0) -- (0,7.8) node[above] {};

\foreach \i\j in {1.5/0, 2.0/0, 3.5/0, 4.5/0, 5.0/0, 6.0/0, 6.5/0, 7.0/0, 7.5/0, 0.5/1.0, 0.5/1.5, 0.5/3.5, 0.5/5.5, 1.0/3.0, 1.0/3.5, 1.0/5.0, 1.0/6.5, 1.0/7.0, 1.0/7.5, 1.5/0.5, 1.5/1.5, 1.5/2.5, 1.5/3.0, 1.5/4.5, 1.5/6.5, 1.5/7.0, 1.5/7.5, 2.0/0.5, 2.0/1.5, 2.0/2.0, 2.0/2.5, 2.0/3.0, 2.0/3.5, 2.0/4.0, 2.0/6.5, 2.0/7.0, 2.5/0.5, 2.5/1.0, 2.5/1.5, 2.5/2.0, 2.5/3.5, 2.5/4.5, 2.5/5.0, 2.5/7.0, 3.0/0.5, 3.0/1.0, 3.0/1.5, 3.0/2.0, 3.0/2.5, 3.0/3.0, 3.0/3.5, 3.0/5.0, 3.0/5.5, 3.0/6.0, 3.0/6.5, 3.0/7.5, 3.5/1.0, 3.5/2.0, 3.5/3.0, 3.5/3.5, 3.5/5.5, 3.5/6.0, 3.5/6.5, 3.5/7.5, 4.0/0.5, 4.0/2.5, 4.0/3.0, 4.0/3.5, 4.0/4.0, 4.0/5.0, 4.0/6.5, 4.0/7.0, 4.0/7.5, 4.5/0.5, 4.5/1.0, 4.5/1.5, 4.5/2.0, 4.5/3.0, 4.5/4.0, 4.5/4.5, 4.5/5.0, 4.5/6.0, 4.5/7.0, 5.0/0.5, 5.0/1.0, 5.0/1.5, 5.0/2.0, 5.0/3.0, 5.0/3.5, 5.0/4.0, 5.0/7.0, 5.5/2.0, 5.5/3.0, 5.5/4.0, 5.5/5.0, 5.5/6.0, 5.5/6.5, 5.5/7.0, 5.5/7.5, 6.0/1.0, 6.0/1.5, 6.0/2.0, 6.0/2.5, 6.0/3.5, 6.0/4.0, 6.0/4.5, 6.5/0.5, 6.5/1.0, 6.5/2.0, 6.5/2.5, 6.5/4.0, 6.5/5.0, 6.5/6.0, 6.5/6.5, 6.5/7.5, 7.0/4.5, 7.0/5.5, 7.0/6.5, 7.0/7.5, 7.5/0.5, 7.5/3.0, 7.5/4.0, 7.5/4.5, 7.5/6.0, 7.5/6.5, 7.5/7.0, 7.5/7.5}{
\draw[-, color=black, line width=0.5mm] (\i-0.5,\j) -- (\i,\j);
	}
\foreach \i\j in {0/1.0, 0/3.0, 0/4.0, 0/4.5, 0/5.0, 0/5.5, 0/6.0, 0/6.5, 0/7.0, 0/7.5, 0.5/0.5, 0.5/2.0, 0.5/2.5, 0.5/3.0, 0.5/4.0, 0.5/4.5, 0.5/5.0, 0.5/6.0, 0.5/6.5, 0.5/7.0, 0.5/7.5, 1.0/0.5, 1.0/1.0, 1.0/1.5, 1.0/2.0, 1.0/2.5, 1.0/4.0, 1.0/4.5, 1.0/5.5, 1.0/6.0, 1.5/1.0, 1.5/2.0, 1.5/3.5, 1.5/4.0, 1.5/5.0, 1.5/5.5, 1.5/6.0, 2.0/1.0, 2.0/4.5, 2.0/5.0, 2.0/5.5, 2.0/6.0, 2.0/7.5, 2.5/2.5, 2.5/3.0, 2.5/4.0, 2.5/5.5, 2.5/6.0, 2.5/6.5, 2.5/7.5, 3.0/4.0, 3.0/4.5, 3.0/7.0, 3.5/0.5, 3.5/1.5, 3.5/2.5, 3.5/4.0, 3.5/4.5, 3.5/5.0, 3.5/7.0, 4.0/1.0, 4.0/1.5, 4.0/2.0, 4.0/4.5, 4.0/5.5, 4.0/6.0, 4.5/2.5, 4.5/3.5, 4.5/5.5, 4.5/6.5, 4.5/7.5, 5.0/2.5, 5.0/4.5, 5.0/5.0, 5.0/5.5, 5.0/6.0, 5.0/6.5, 5.0/7.5, 5.5/0.5, 5.5/1.0, 5.5/1.5, 5.5/2.5, 5.5/3.5, 5.5/4.5, 5.5/5.5, 6.0/0.5, 6.0/3.0, 6.0/5.0, 6.0/5.5, 6.0/6.0, 6.0/6.5, 6.0/7.0, 6.0/7.5, 6.5/1.5, 6.5/3.0, 6.5/3.5, 6.5/4.5, 6.5/5.5, 6.5/7.0, 7.0/0.5, 7.0/1.0, 7.0/1.5, 7.0/2.0, 7.0/2.5, 7.0/3.0, 7.0/3.5, 7.0/4.0, 7.0/5.0, 7.0/6.0, 7.0/7.0, 7.5/1.0, 7.5/1.5, 7.5/2.0, 7.5/2.5, 7.5/3.5, 7.5/5.0, 7.5/5.5}{
\draw[-, color=black, line width=0.5mm] (\i,\j-0.5) -- (\i,\j);
	}
\end{tikzpicture}
\caption{A portion of the random walk web obtained from the $B_{ij}$'s. If $B_{ij}=0$, an edge is placed from $(i-1,j)$ to $(i,j)$; otherwise, the edge is placed from $(i,j-1)$ to $(i,j)$.}
\label{fig: RW web}
\end{figure}
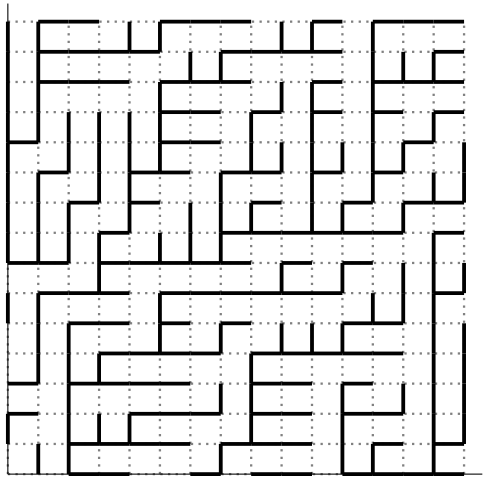

\section{Main results}

Let $F_H(m,n)$ be the first-passage value at $(m,n)$ in the SJ model with weights $B_{ij}\xi_{ij}$ on horizontal edges $(i-1,j) \to (i,j)$ and $F_V(m,n)$ the first-passage value at $(m,n)$ for the SJ model on vertical edges with weights $(1-B_{ij})\eta_{ij}$ (that is, we change all the weights on vertical edges to be 0 for $F_H$ and all the weights on horizontal edges to be 0 for $F_V$). Let $f_H$ and $f_V$ be the corresponding limit shapes. Clearly, we have $F_H(m,n) \le F(m,n)$, since we can simply ignore any weights picked up along vertical edges when following a geodesic for $F$ and this gives a valid path for $F_H$. Similarly, we also have $F_V(m,n) \le F(m,n)$. Passing to the limit, this implies the corresponding bound on limit shapes:
\begin{equation}\label{eq: lower bound LS}
\max(f_H(x,y), f_V(x,y)) \le f(x,y)
\end{equation}
for all $x,y \ge 0$. Our first result is that equality holds in \eqref{eq: lower bound LS}.
\begin{theorem}\label{th: LS thm}
\[\max(f_H(x,y), f_V(x,y)) = f(x,y)\]
\end{theorem}
As we will see later, $f_H(x,y)=0$ for all $x \le (1-p)y/p$ and $f_V(x,y)=0$ for all $x \ge (1-p)y/p$, so Theorem \ref{th: LS thm} can be rewritten as
\[f(x,y)=
\begin{cases}
f_H(x,y) & \quad \text{if } x \ge \frac{1-p}{p}y\\
f_V(x,y) & \quad \text{if } x < \frac{1-p}{p}y.
\end{cases}
\]
Since the scaled first-passage functions converge to their limit shape, it follows that
\begin{equation}\label{eq: error term}
F(\lfloor nx \rfloor,\lfloor ny \rfloor)=F_H(\lfloor nx \rfloor,\lfloor ny \rfloor)+o(n)
\end{equation}
for any $x > (1-p)y/p$, and a similar asymptotic holds with $F_V$ when $x < (1-p)y/p$. Thus the plane is divided into two regions, one where essentially only the horizontal weights matter, and one where only the vertical weights matter. Numerical simulations seem to show that the error term is actually of much smaller order than $o(n)$, and that the equality $\max(f_H(x,y),f_V(x,y))=f(x,y)$ almost holds even in the prelimit.

Next we consider the special case where at least one of the $\xi_{ij}$'s or $\eta_{ij}$'s are Bernoulli distributed (without loss of generality, we can assume it is the $\xi_{ij}$'s). In this situation, we can show that the error in \eqref{eq: error term} is in fact $o(n^{1/3})$. As a consequence, we obtain the following limiting fluctuations result.
\begin{theorem}\label{th: fluctuations thm}
Suppose that $B_{ij} \xi_{ij}$ follow a Bernoulli($p$) distribution. Assume $0<p<1$, and let $x,y$ be positive and satisfying $x>(1-p)y/p$. Define 
\[\chi(x,y)=\left[\frac{p(1-p)}{xy}(\sqrt{px}-\sqrt{(1-p)y})^2(\sqrt{(1-p)x}+\sqrt{py})^2\right]^{1/3}.\]
Then
\begin{equation}\label{eq: fluctuations limit}
\frac{F(nx,ny)-nf(x,y)}{\chi(x,y)n^{1/3}} \Rightarrow TW_{GUE}
\end{equation}
Here $TW_ {GUE}$ is the Tracy--Widom GUE distribution.
\end{theorem}
Theorem \ref{th: fluctuations thm} is known for the SJ model; this was proven by Dauvergne, Nica and Vir\'ag in \cite{dauv-nica-vir23}. In fact, they prove a functional version of this result; we state it here in our language as it will be of use in the proof of Theorem \ref{th: fluctuations thm}.
\begin{theorem}\cite[Corollary 6.11]{dauv-nica-vir23}\label{th: dauv-nica-vir}
Consider the Sepp\"al\"ainen--Johansson model with Bernoulli($p$) weights with $0<p<1$, and let $x,y$ be positive and satisfying $x>(1-p)y/p$. Define 
\[\tau(x,y)=2\left[\frac{x^2}{y\sqrt{p(1-p)}}(\sqrt{px}-\sqrt{(1-p)y})(\sqrt{(1-p)x}+\sqrt{py})\right]^{1/3}\]
\[\chi(x,y)=\left[\frac{p(1-p)}{xy}(\sqrt{px}-\sqrt{(1-p)y})^2(\sqrt{(1-p)x}+\sqrt{py})^2\right]^{1/3}\]
\[\rho(x,y)=p-\sqrt{\frac{p(1-p)y}{x}}.\]
Then
\begin{equation}\label{eq: fluctuations limit SJ}
\frac{F_H(nx+\tau(x,y)n^{2/3}t,ny)-nf_H(x,y)-\tau(x,y)n^{2/3}t \rho(x,y)}{\chi(x,y)n^{1/3}} \Rightarrow \Ac_1(t)
\end{equation}
with respect to uniform convergence on compact sets. Here $\Ac_1(t)$ is the top line of the Airy line ensemble. 
\end{theorem}
The \emph{Airy line ensemble} is a collection of continuous stochastic processes $\Ac_1, \Ac_2, \dots$ with the property that 
\begin{equation}\label{eq: airy line ensemble}
\Ac_1(t) \ge \Ac_2(t) \ge \dots
\end{equation}
for all $t$. Pr\"afer and Spohn first constructed a process with the same finite distributions in \cite{pra-spo02} as the scaling limit of the polynuclear growth model, however it was only shown in \cite{cor-ham14} that there is a version of this process with continuous sample paths that satisfies \eqref{eq: airy line ensemble}. That is, the processes $\Ac_1, \Ac_2, \dots$ can be coupled together in such a way that they are almost surely non-intersecting. The top line $\Ac_1$ is usually called the Airy$_2$ process, and is the scaling limit of the largest eigenvalue in the Dyson Brownian motion; see \cite{johan03}. In particular, $\Ac_1(0)$ follows the Tracy--Widom GUE distribution. For our purposes, this last fact and the continuity of the paths of $\Ac_1$ are the only properties of the Airy$_2$ process that we will need.

Setting $t=0$ in \eqref{eq: fluctuations limit SJ} shows that the fluctuations of the SJ first-passage value converge in distribution to the Tracy--Widom GUE distribution. Our contribution is to extend this result to the generalized SJ model. The scaling exponents and coefficients in \eqref{eq: fluctuations limit SJ} are exactly the same in both models, as well as the limiting distribution. This further reinforces this behaviour of $F$ being almost completely determined by only one set of weights on either side of the critical line.

It is believed that \eqref{eq: fluctuations limit SJ} should be universal, in the sense that the scaling exponents and limiting distribution should be the same for any choice of distribution of the weights, and thus the SJ model should be in the Kardar--Parisi--Zhang (KPZ) universality class. As the proof of Theorem \ref{th: fluctuations thm} will show, for any set of weights for which \eqref{eq: fluctuations limit SJ} holds for the SJ model, the convergence of the marginal at 0 also holds in the generalized case.

\section{Proof of Theorem \ref{th: LS thm}}

Throughout the next two sections, we will usually simplify notation by omitting the floor function when evaluating $F$ or $F_H$ at non-integer points. It is then understood that if $a$ and $b$ are not integers, then $F(a,b)$ is defined to be $F(\lfloor a \rfloor, \lfloor b \rfloor)$. The main ingredient which goes in the proof of Theorem \ref{th: LS thm} is the following curious identity.

 \begin{lemma}\label{le: boundary condition lem}
Let $B_{ij} \in \{0,1\}$, and let $\xi_{ij}, \eta_{ij}$ be collections of real numbers. Let $F(m,n)$ be the first passage value from $(0,0)$ to $(m,n)$ with weights $\om_{ij}=B_{ij}\xi_{ij}$ on horizontal edges $(i-1,j) \to (i,j)$ and weights $\tom_{ij}=(1-B_{ij})\eta_{ij}$ on vertical edges $(i,j-1) \to (i,j)$, and let $F_H(m,n)$ be the first passage value from $(0,0)$ to $(m,n)$ where the weights on the horizontal edges are $\om_{ij}$ and the weights on the vertical edges are 0 except for the weights on the $y$-axis which are $\tom_{0j}$. Suppose that
\[\xi_{ij} \ge 0 \quad \text{for all }i,j \ge 0,\]
\[\eta_{ij} \ge 0 \quad \text{for all } i \ge 1, j \ge 0,\]
\[\eta_{0j} \le 0 \quad \text{for all } j \ge 0,\]
that is the $\xi_{ij}$'s and $\eta_{ij}$'s are all non-negative except for the $\eta_{ij}$'s lying on the $y$-axis which are all non-positive. Then
\[F(m,n)=F_H(m,n)\]
for all $m,n$.
 \end{lemma}
 
 Thus if one replaces the weights on the $y$-axis with non-positive weights, then the first-passage value corresponds exactly to the first-passage value on horizontal edges. That is, we can completely ignore any vertical edges not on the $y$-axis; given a geodesic from $(0,0)$ to $(m,n)$, it cannot pass through a vertical edge of non-zero weight except for edges on the $y$-axis. Lemma \ref{le: boundary condition lem} is deterministic and holds for arbitrary collections of numbers $\xi_{ij}$, $\eta_{ij}$ and $B_{ij}$ satisfying the conditions in the lemma. A similar result holds with weights on the $x$-axis changed to being non-positive and first-passage on vertical edges.

\begin{proof}
Let $X_{ij}$, $Y_{ij}$ be the horizontal and vertical increments for $F$:
\begin{align*}
X_{ij}&=F(i,j)-F(i-1,j)\\
Y_{ij}&=F(i,j-1)-F(i,j)
\end{align*}
and define $X^H_{ij}$ and $Y^H_{ij}$ similarly for $F_H$. If we're given all the increments of a model, we can deduce what the first passage values are by just adding/subtracting the increments:
\[F(m,n)=\sum_{i=1}^m X_{i0}-\sum_{j=1}^n Y_{mj}\]
(and similarly for $F_H$). It is therefore enough to show that $X_{ij}=X^H_{ij}$ and $Y_{ij}=Y^H_{ij}$ for all $i,j$. Since any path from $(0,0)$ to $(i,j)$ must pass through exactly one of the vertices $(i-1,j)$ or $(i,j-1)$, it is easy to see that $F$ and $F_H$ satisfy the recursions
\begin{equation}\label{eq: fp recursion}
\begin{aligned}
F(i,j)&=\min(F(i-1,j)+\om_{ij},F(i,j-1)+\tom_{ij})\\
F_H(i,j)&=\min(F_H(i-1,j)+\om_{ij},F_H(i,j-1)).
\end{aligned}
\end{equation}
Using \eqref{eq: fp recursion}, we obtain recursions for $X_{ij}$ and $Y_{ij}$:
\begin{align*}
X_{ij}&=\min(F(i-1,j)+\om_{ij},F(i,j-1)+\tom_{ij})-F(i-1,j)\\
&=\min(\om_{ij},X_{i,j-1}+Y_{i-1,j}+\tom_{ij})
\end{align*}
and
\begin{align*}
Y_{ij}&=F(i,j-1)-\min(F(i-1,j)+\om_{ij},F(i,j-1)+\tom_{ij})\\
&=\max(X_{i,j-1}+Y_{i-1,j}-\om_{ij},-\tom_{ij}).
\end{align*}
Now, suppose that $X_{i,j-1}$ and $Y_{i-1,j}$ are both non-negative. By definition of the model, at least one of $\om_{ij}$ or $\tom_{ij}$ must be 0. If $\tom_{ij}=0$, then $Y_{ij}=\max(X_{i,j-1}+Y_{i-1,j}-\om_{ij},0)$. If $\om_{ij}=0$, then because $X_{i,j-1}+Y_{i-1,j}$ and $\tom_{ij}$ are non-negative,
\begin{align*}
Y_{ij}&=\max(X_{i,j-1}+Y_{i-1,j},-\tom_{ij})=X_{i,j-1}+Y_{i-1,j}\\
&=\max(X_{i,j-1}+Y_{i-1,j}-\om_{ij},0).
\end{align*}
Likewise, if both $X_{i,j-1}$ and $Y_{i-1,j}$ are non-negative and $\tom_{ij}=0$, then 
\[X_{ij}=\min(\om_{ij}, X_{i,j-1}+Y_{i-1,j}),\]
and if instead $\om_{ij}=0$, then
\[X_{ij}=\min(0, X_{i,j-1}+Y_{i-1,j}+\tom_{ij})=0=\min(\om_{ij}, X_{i,j-1}+Y_{i-1,j}).\]
Note also from these recursions that $X_{ij}$ and $Y_{ij}$ are then both non-negative. Using \eqref{eq: fp recursion} for $F_H$ and the fact that vertical edges have weight 0 in this model, we find 
\[X^H_{ij}=\max(0,X^H_{i,j-1}+Y^H_{i-1,j}-\om_{ij})\]
and
\[Y^H_{ij}=\min(X^H_{i,j-1}+Y^H_{i-1,j},\om_{ij}).\]
So the increments for $F$ and $F_H$ satisfy the exact same recursion provided the increments for $F$ are non-negative. The increments are indeed non-negative; this is clear when $i=0$ or $j=0$ since $\xi_{ij} \ge 0$ and $\eta_{i0} \le 0$ for all $i$ and $j$, and the general case follows by double induction on $(i,j)$. Finally the boundary conditions are the same since the weights on the edges of both axes are the same.
\end{proof}

We apply Lemma \ref{le: boundary condition lem} to the special case where the weights on the $y$-axis are zero. Since the first-passage value is the same as $F_H(m,n)$ when we change all the weights on the $y$-axis to be zero, there must be a geodesic $\pi$ for $F_H(m,n)$ which does not pass through any vertical edge of non-zero weight except possibly on the $y$-axis. We obviously have $S(\pi) \ge F(m,n)$, and the only difference between these two are the extra weights picked up by $\pi$ along the $y$-axis:
\[F(m,n) \le S(\pi)=F_H(m,n)+\sum_{j=1}^Z (1-B_{0j})\eta_{0j}\]
where $Z=\min\{k \ge 0: (1,k) \in \pi\}$ is the position where $\pi$ exits the $y$-axis (when $Z=0$, we interpret the sum as being 0). Since $\pi$ is a geodesic for $F_H(m,n)$, we have $Z \le \Dc(m,n)$, where $\Dc(m,n)$ is the top-most departure point from the $y$-axis that a geodesic for $F_H$ can take:
\[\Dc(m,n)=\max\{k \ge 0: F_H(0,k;m,n)=F_H(m,n)\}.\]
Note that $\Dc(m,n)$ is independent of the $\eta_{ij}$'s since it is defined completely in terms of first-passage percolation with the weights $B_{ij}\xi_{ij}$ on horizontal edges. Thus, along with the inequality $(1-B_{0j}) \le 1$, we have the upper bound
\begin{equation}\label{eq: upper bound for F}
F(m,n) \le F_H(m,n)+\sum_{j=1}^{\Dc(m,n)} \eta_{0j}.
\end{equation}
Together with the obvious lower bound $F(m,n) \ge F_H(m,n)$, Theorem \ref{th: LS thm} will be proved for $y \le (1-p)x/p$ provided that we can show
\begin{equation}\label{eq: sum with random number of terms}
\frac{1}{n} \sum_{j=1}^{\Dc(nx,ny)} \eta_{0j} \to 0
\end{equation}
in probability (this will imply $F(nx,ny)/n \to f_H(x,y)$ in probability, but since we know that $F(nx,ny)/n \to f(x,y)$ almost surely by the subadditive ergodic theorem and almost sure convergence implies convergence in probability, we thus have $F(nx,ny)/n \to f_H(x,y)$ almost surely). As mentioned above, $\Dc(m,n)$ is independent of the $\eta_{ij}$'s, so by the law of large numbers, \eqref{eq: sum with random number of terms} is equivalent to 
\begin{equation}\label{eq: D conv prob}
\frac{\Dc(nx,ny)}{n} \to 0
\end{equation}
in probability. It will be more convenient to work with the bottom-most entry point $\Ec(m,n)$ to the line $x=m$:
\[\Ec(m,n)=\max\{k \ge 0: F_H(m,n-k)=F_H(m,n)\}.\]
We can see that $\Ec(m,n)$ has the same law as $\Dc(m,n)$, since it corresponds exactly to the bottom-most departure point for first-passage percolation with down-left paths from $(m,n)$ to $(0,0)$. So \eqref{eq: D conv prob} is equivalent to showing that
\begin{equation}\label{eq: E conv prob}
\frac{\Ec(nx,ny)}{n} \to 0
\end{equation}
in probability. To get Theorem \ref{th: LS thm} for $x<(1-p)y/p$, one can use the same argument as above but with first-passage percolation on vertical edges instead and by considering the right-most departure from the $x$-axis and left-most entry to the line $y=n$. The proof for this case is exactly the same so we will not write it down.

\begin{lemma}\label{le: positive limit shape}
Let $\om_{ij}$ be i.i.d. and non-negative, and let $f_H$ be the limit shape for the SJ model on horizontal edges with weights $\om_{ij}$. Let $p=\PP(\om_{ij}=0)$. Then $f_H(x,y)>0$ if and only if $x>\frac{p}{1-p}y$.
\end{lemma}

\begin{proof}
First assume that $0<x<py/(1-p)$ (the case where $x=0$ is trivial, since $F_H(0,n)=0$ for all $n$). Note in particular that this implies $p>0$ in that case. In order for $F_H(m,n)$ to be 0, there has to be a path from the origin to $(m,n)$ which only visits edges of weight 0. Every time the path sees a horizontal edge of weight 0, it will take it, otherwise it will keep moving up until it sees an edge of weight 0. The number of up steps it needs to take before it sees such an edge has the geometric distribution on $\{0,1, \dots\}$ with probability of success $p$, and it needs to take $m$ right steps. So
\[\PP(F_H(m,n)=0)=\PP(Z_1+\dots+Z_m \le n)\]
where $Z_1, \dots, Z_m$ are i.i.d. $\text{Geo}(p)$ random variables. Take $\lfloor nx \rfloor$ and $\lfloor ny \rfloor$ instead of $m$ and $n$. Then for $\theta>0$, we have, by Markov's inequality,
\begin{align*}
\PP(F_H(\lfloor nx \rfloor, \lfloor ny \rfloor) \neq 0)&=\PP(Z_1+\dots+Z_{\lfloor nx \rfloor} > \lfloor ny \rfloor)\\
&=\PP(e^{\theta(Z_1+\dots+Z_{\lfloor nx \rfloor})}>e^{\theta \lfloor ny \rfloor}) \le \frac{\EE(e^{\theta Z_1})^{\lfloor nx \rfloor}}{e^{\theta \lfloor ny \rfloor}}\\
&=\exp(\lfloor nx \rfloor \log p-\lfloor nx \rfloor \log(1-(1-p)e^{\theta})-\theta \lfloor ny \rfloor).
\end{align*}
There is some $0<\ep<1$ such that $x=(1-\ep)py/(1-p)$. Take 
\[\theta:=\log\left(\frac{y}{(1-p)(x+y)}\right).\]
By this condition on $x$ and $y$, we have $\theta=-\log(1-\ep p)>0$. Substituting $\theta$ in the above, we then find
\begin{align*}
&\PP(F_H(\lfloor nx \rfloor, \lfloor ny \rfloor) \neq 0)\\
&\le \exp\left(\lfloor nx \rfloor \log p-\lfloor nx \rfloor \log\left(1-\frac{(1-p)}{1-\ep p}\right)+\lfloor ny \rfloor \log(1-\ep p)\right)\\
&=\exp((\lfloor nx \rfloor +\lfloor ny \rfloor )\log(1-\ep p)-\lfloor nx \rfloor \log(1-\ep))\\
&\le \exp((nx+ny-2)\log(1-\ep p)-nx \log(1-\ep))\\
&=\left(\frac{1}{1-\ep p}\right)^2 \exp \left(nx \left[ \frac{1-\ep p}{p(1-\ep)}\log(1-\ep p)-\log(1-\ep)\right] \right).
\end{align*}
Let $g(\ep)$ be the expression in square brackets above. Then
\[g'(\ep)=\frac{(1-p)\log(1-\ep p)}{p(1-\ep)^2}<0\]
so $g$ is strictly decreasing. Since $g(0)=0$, it follows that $g(\ep)<0$ for every $0<\ep<1$, and therefore the above probabilities are summable in $n$. By the Borel--Cantelli lemma, it follows that $F_H(\lfloor nx \rfloor, \lfloor ny \rfloor)$ is 0 for all but finitely many $n$ almost surely, and thus the limit shape must satisfy $f_H(x,y)=0$. This proves that $f_H(x,y)=0$ for all $x<py/(1-p)$, and by continuity, we obtain this for $x=py/(1-p)$ as well.

Finally assume $x>py/(1-p)$, and choose $q>p$ such that $x>qy/(1-q)$. Since
\[1-p=\PP(\om_{ij}>0)=\lim_{s \downarrow 0} \PP(\om_{ij}>s),\]
there is some $c>0$ such that $\PP(\om_{ij}>c) > 1-q$, or equivalently that $\PP(\om_{ij} \le c)<q$. Define new weights $\tom_{ij}$ as follows:
\[\tom_{ij}=
\begin{cases}
0 & \quad \text{if } \om_{ij} \le c\\
c& \quad \text{if } \om_{ij}>c.
\end{cases}
\]
Then we have $\om_{ij} \ge \tom_{ij}$ for all $i,j$, and $\tom_{ij}$ is $c$ times a Bernoulli$(1-s)$ for some $0 \le s< q$. With $\widetilde{f}_H$ the limit shape of the $\tom_{ij}$'s (which we can compute explicitly using \eqref{eq: Bernoulli formula}), we then have
\[f_H(x,y) \ge \widetilde{f}_H(x,y)=c(\sqrt{(1-s)x}-\sqrt{sy})^2>0. \qedhere \]
\end{proof}

\begin{lemma}\label{le: strictly decreasing limit shape}
With the hypotheses of the  previous lemma and the assumption that $p>0$, we have for any fixed $x$ that the function $y \mapsto f_H(x,y)$ is strictly decreasing on $[0,(1-p)x/p]$.
\end{lemma}

\begin{proof}
Let $0 \le y_1<y_2 < (1-p)x/p$, and pick any $z>(1-p)x/p$. Then there is a $t \in (0,1)$ such that $y_2=(1-t)y_1+tz$. By Lemma \ref{le: positive limit shape}, $f_H(x,y_1)>0$ and $f_H(x,z)=0$. Since $y \mapsto f_H(x,y)$ is convex, it follows that
\begin{align*}
f_H(x,y_2)&=f_H(x,(1-t)y_1+tz) \le (1-t)f_H(x,y_1)+tf_H(x,z)\\
&=(1-t)f_H(x,y_1)<f_H(x,y_1). \qedhere
\end{align*}
\end{proof}

We are now ready to show \eqref{eq: E conv prob} and conclude the proof of Theorem \ref{th: LS thm}. Let $q=\PP(B_{ij}\xi_{ij}=0)$. Then for a point $(x,y)$ satisfying $x>qy/(1-q)$, we have that the function $z \mapsto f_H(x,z)$ is strictly decreasing in a neighbourhood of $y$ by Lemma \ref{le: strictly decreasing limit shape}. So for all $\ep>0$ small enough, $f_H(x,y-\ep)>f_H(x,y)$, and because 
\[\frac{F_H(nx,ny)}{n} \to f_H(x,y), \quad \frac{F_H(nx,n(y-\ep))}{n} \to f_H(x,y-\ep)\]
almost surely, it follows that
\[\PP(\Ec(nx,ny)\ge n\ep)=\PP(F_H(nx,ny)=F_H(nx,n(y-\ep))) \to 0.\]
By the remarks leading up to \eqref{eq: E conv prob}, this shows that $f(x,y)=f_H(x,y)$ for $x>qy/(1-q)$. As explained previously, a symmetric argument using first-passage percolation on vertical edges shows that we also have $f(x,y)=f_V(x,y)$ for $x<(1-r)y/r$ where $r=\PP((1-B_{ij})\eta_{ij}=0)$. By continuity of the limit shape and Lemma \ref{le: positive limit shape}, we also have $f(x,y)=f_H(x,y)=0$ for $x=qy/(1-q)$ and $f(x,y)=f_V(x,y)=0$ for $x=(1-r)y/r$.

In the case where both $\xi_{ij}$ and $\eta_{ij}$ are positive almost surely, we are done (since then we have $q=1-p$ and $r=p$). If not, then we still have to deal with points in between the lines $x=qy/(1-q)$ and $x=(1-r)y/r$. However on those lines, we have $f(x,y)=0$, and $f$ is non-negative and convex. So $f$ must also be zero in between those lines, and therefore equals both $f_H$ and $f_V$ there. This shows that $f(x,y)=\max(f_H(x,y),f_V(x,y))$ in all cases and concludes the proof of Theorem \ref{th: LS thm}!

\section{Proof of Theorem \ref{th: fluctuations thm}}

We will abbreviate things in the statement of Theorem \ref{th: dauv-nica-vir} by writing $\tau=\tau(x,y)$, $\chi=\chi(x,y)$ and $\rho=\rho(x,y)$. Write $\Fc_n(t)$ for the left-hand side of \eqref{eq: fluctuations limit SJ}. Thus we have $\Fc_n(t) \Rightarrow \Ac_1(t)$ with respect to uniform convergence on compact sets.

By the Skorokhod representation theorem \cite[Theorem 6.7]{bil99}, the $\Fc_n$ and $\Ac_1$ can be coupled together on the same probability space such that $\Fc_n \to \Ac_1$ uniformly on compact sets, almost surely. We henceforth work with this particular coupling. Fix $\ep>0$ and define
\[t_n=\frac{\frac{\ep}{y} n^{1/3}x}{\tau(n-\frac{\ep}{y}n^{1/3})^{2/3}}.\]
Then 
\[\Fc_{n-\frac{\ep}{y}n^{1/3}}(t_n) \to \Ac_1(0)\]
(we again use the convention that $\Fc_{k}=\Fc_{\lfloor k \rfloor}$ when $k$ is not an integer). Indeed, let $K$ be a compact subset of $\RR$ which contains all the $t_n$'s. Then 
\[|\Fc_{n-\frac{\ep}{y}n^{1/3}}(t_n)-\Ac_1(0)| \le \sup_{s \in K}|\Fc_{n-\frac{\ep}{y}n^{1/3}}(s)-\Ac_1(s)|+|\Ac_1(t_n)-\Ac_1(0)|.\]
The first term on the right-hand side above converges to 0 since $\Fc_n \to \Ac_1$ uniformly on $K$, and the second term converges to 0 because $\Ac_1$ is continuous. We have
\begin{align*}
\Fc_{n-\frac{\ep}{y}n^{1/3}}(t_n)=\frac{F_H(nx,ny-\ep n^{1/3})-(n-\frac{\ep}{y}n^{1/3})f(x,y)-\frac{\ep}{y}xn^{1/3} \rho}{(n-\frac{\ep}{y}n^{1/3})^{1/3} \chi},
\end{align*}
and so
\begin{equation}\label{eq: F at lower point}
F_H(nx,ny-\ep n^{1/3})-(n-\frac{\ep}{y}n^{1/3})f(x,y)-\frac{\ep}{y}x \rho n^{1/3}=\chi \Ac_1(0) n^{1/3}+o(n^{1/3}).
\end{equation}
Since $\Fc_{n-\frac{\ep}{y}n^{1/3}}(0) \to \Ac_1(0)$, we also have
\begin{equation}\label{eq: F at upper point}
F_H(nx, ny)-nf(x,y)=\chi \Ac_1(t) n^{1/3}+o(n^{1/3}).
\end{equation}
Now subtract \eqref{eq: F at upper point} from \eqref{eq: F at lower point} and divide by $n^{1/3}$. After some rearranging, this yields
\begin{equation}\label{eq: F at points n^1/3 apart}
\frac{F_H(nx,ny-\ep n^{1/3})-F_H(nx,ny)}{n^{1/3}}=\frac{\ep}{y}(x\rho-f(x,y))+o(1).
\end{equation}
Since we are in the Bernoulli case, $f(x,y)$ is given by \eqref{eq: Bernoulli formula}, and so we find
\begin{align*}
x\rho(x,y)-f(x,y)&=px-\sqrt{p(1-p)xy}-(\sqrt{px}-\sqrt{(1-p)y})^2\\
&=\sqrt{p(1-p)xy}-(1-p)y>0
\end{align*}
by our assumption on $x,y$ and $p$. This together with \eqref{eq: F at points n^1/3 apart} implies that for all sufficiently large $n$, we must have
\[F_H(nx,ny) \neq F_H(nx,ny-\ep n^{1/3}).\]
Consequently, the bottom-most entry point $\Ec(nx,ny)$ is at most $\ep n^{1/3}$, and this implies
\[\limsup_{n \to \infty} \frac{\Ec(nx,ny)}{n^{1/3}} \le \ep\]
almost surely. Since $\ep$ was arbitrary, it follows that
\[\frac{\Ec(nx,ny)}{n^{1/3}} \to 0\]
almost surely, and because $\Dc(nx,ny)$ has the same distribution as $\Ec(nx,ny)$, we deduce that
\[\frac{\Dc(nx,ny)}{n^{1/3}} \to 0\]
in probability. By the law of large numbers, it then follows that
\begin{equation}\label{eq: sum up to D conv in prob}
\frac{1}{n^{1/3}} \sum_{j=1}^{\Dc(nx,ny)} \eta_{0j} \to 0
\end{equation}
in probability. Let us note that it is fairly straightforward to generalize the above argument to obtain that
\begin{equation}\label{eq: uniform conv of E}
\frac{\Ec(nx+\tau n^{2/3} t,ny)}{n^{1/3}} \to 0
\end{equation}
uniformly for $t$ in a compact set almost surely. However, while it is true that $\Ec(m,n)$ and $\Dc(m,n)$ have the same distribution for a \emph{fixed} endpoint $(m,n)$, it is not the case that the joint laws of $\{\Ec(m,k): k \in S\}$ and $\{\Dc(m,k): k \in S\}$ are the same for $k$ varying in some set of integers $S$. We therefore cannot conclude that \eqref{eq: uniform conv of E} holds for $\Dc$; this is the only obstacle in obtaining uniform convergence as in \eqref{eq: fluctuations limit SJ} for the generalized SJ model.

We are now ready to conclude. Let $G$ be a bounded, uniformly continuous function on $\RR$, and let $\ep>0$. Then there is a $\delta>0$ such that $|G(x)-G(y)|<\ep$ whenever $|x-y|<\delta$. By \eqref{eq: upper bound for F},
\begin{align*}
\frac{F_H(nx,ny)-nf(x,y)}{\chi n^{1/3}} &\le \frac{F(nx,ny)-nf(x,y)}{\chi n^{1/3}}\\
 &\le \frac{F_H(nx,ny)-nf(x,y)}{\chi n^{1/3}}+\frac{1}{\chi n^{1/3}} \sum_{j=1}^{\Dc(nx,ny)} \eta_{0j}.
\end{align*}
Thus if the scaled fluctuations of $F(nx,ny)$ and $F_H(nx,ny)$ are at least $\delta$ apart from each other, then
\[\frac{1}{\chi n^{1/3}} \sum_{j=1}^{\Dc(nx,ny)} \eta_{0j} \ge \delta.\]
Hence,
\begin{align*}
&\EE \left| G\left(\frac{F(nx,ny)-nf(x,y)}{\chi n^{1/3}}\right)-G\left(\frac{F_H(nx,ny)-nf(x,y)}{\chi n^{1/3}}\right)\right| \\
&\le 2\sup|G| \,\PP\left( \frac{1}{\chi n^{1/3}} \sum_{j=1}^{\Dc(nx,ny)} \eta_{0j} \ge \delta\right)+\ep \rightarrow \ep
\end{align*}
by \eqref{eq: sum up to D conv in prob}. Since $\ep$ was arbitrary, it then follows along with Theorem \ref{th: dauv-nica-vir} applied with $t=0$ that
\begin{align*}
\lim_{n \to \infty} \EE \left[G\left(\frac{F(nx,ny)-nf(x,y)}{\chi n^{1/3}}\right)\right]&=\lim_{n \to \infty} \EE \left[ G\left(\frac{F_H(nx,ny)-nf(x,y)}{\chi n^{1/3}}\right)\right]\\
&=\int_{-\infty}^{\infty} G(x) f_{GUE}(x) dx
\end{align*}
where $f_{GUE}$ is the Tracy--Widom GUE density. This concludes the proof of Theorem \ref{th: fluctuations thm}!

\bibliographystyle{amsplain}
\bibliography{bibliography_SJ}

\end{document}